\documentclass[12pt,reqno]{amsart}
\usepackage[headings]{fullpage}
\usepackage{amssymb,amsmath,mathtools,bbm,tikz,tikz-cd,color,relsize,multirow}
\usepackage[all,cmtip]{xy}
\usepackage{url}
\usepackage{soul}
\usepackage[group-separator={,},group-minimum-digits={3}]{siunitx}
\usetikzlibrary{positioning}
\usetikzlibrary{calc}
\usetikzlibrary{decorations.markings}
\usetikzlibrary{arrows}
\usetikzlibrary{calc}
\usetikzlibrary{decorations.markings}
\tikzstyle{hvector}=[inner sep=2pt,draw=blue!50,fill=blue!10,thick]
\tikzstyle{unit}=[inner sep=2pt,shape=circle, draw]
\tikzstyle{counit}=[inner sep=2pt,shape=circle, draw,fill=gray]
\tikzstyle{antipode}=[inner sep=2pt,shape=rectangle, draw]
\tikzstyle{cocycle}=[inner sep=2pt,shape=circle, draw]
\tikzstyle{twistedm}=[inner sep=2pt,shape=circle, fill=gray]
\tikzstyle{autom}=[inner sep=2pt,shape=circle, draw]
\tikzstyle{coact}=[inner sep=2pt,shape=circle, fill=black]

\usepackage[bookmarks=true,%
    colorlinks=true,%
    linkcolor=blue,%
    citecolor=blue,%
    filecolor=blue,%
    menucolor=blue,%
    urlcolor=blue,%
    breaklinks=true]{hyperref}
\usepackage{slashed}    
\usepackage{listings}        
\usepackage{verbatim}
\usepackage[normalem]{ulem}  
\usepackage{diagbox}
\usepackage[shortlabels]{enumitem}

\usepackage{tikz,tikz-cd}
\usetikzlibrary{knots, hobby, decorations.pathreplacing,
  shapes.geometric, calc}

\usepackage{orcidlink}

\newtheorem{theorem}{Theorem}[section]
\theoremstyle{definition}
\newtheorem{proposition}[theorem]{Proposition}
\newtheorem{lemma}[theorem]{Lemma}

\newtheorem{corollary}[theorem]{Corollary}

\newtheorem{question}[theorem]{Question}

\def\BZ{\mathbbm Z}
\def\BQ{\mathbbm Q}

\def\be{\begin{equation}}
\def\ee{\end{equation}}

\def\LG{\mathrm{LG}}

\def\Uq{\mathrm{U}_q(\mathfrak{sl}_{2|1})}
\def\fsl{\mathfrak{sl}}
\def\sl21{\mathfrak{sl}_{2|1}}

\def\sdeghV2{\mathrm{sdeg}_\hbar^{V_2}}

\def\KM{\mathrm{KM}}

\definecolor{codegreen}{rgb}{0,0.6,0}
\definecolor{codegray}{rgb}{0.5,0.5,0.5}
\definecolor{codepurple}{rgb}{0.58,0,0.82}
\definecolor{backcolour}{rgb}{0.95,0.95,0.92}

\lstdefinelanguage{PARIGP}{
  keywords={typeof, new, true, false, catch, function, return, null, catch,
    switch, var, if, in, while, do, else, case, break},
  keywordstyle=\color{blue}\bfseries,
  ndkeywords={class, export, boolean, throw, implements, import, this},
  ndkeywordstyle=\color{darkgray}\bfseries,
  identifierstyle=\color{black},
  sensitive=false,
  comment=[l]{//},
  morecomment=[s]{/*}{*/},
  commentstyle=\color{purple}\ttfamily,
  stringstyle=\color{red}\ttfamily,
  morestring=[b]',
  morestring=[b]"
}

\lstdefinestyle{code}{
    backgroundcolor=\color{backcolour},   
    commentstyle=\color{codegreen},
    keywordstyle=\color{magenta},
    numberstyle=\tiny\color{codegray},
    stringstyle=\color{codepurple},
    basicstyle=\ttfamily\footnotesize,
    breakatwhitespace=false,         
    breaklines=true,                 
    captionpos=b,                    
    keepspaces=true,                 
    numbers=left,                    
    numbersep=5pt,                  
    showspaces=false,                
    showstringspaces=false,
    showtabs=false,                  
    tabsize=2
}

\begin{document}

\title[The colored Jones polynomial does not determine the Links--Gould polynomial]{
The colored Jones polynomial does not determine the Links--Gould polynomial}
\author{Stavros Garoufalidis}
\address[Stavros Garoufalidis]{
  International Center for Mathematics, Department of Mathematics \\
  Southern University of Science and Technology \\
  Shenzhen, China \newline
  {\tt \url{http://people.mpim-bonn.mpg.de/stavros}} }
\email{stavros@mpim-bonn.mpg.de}
\author{Shana Yunsheng Li}
\address[Shana Yunsheng Li]{
  Department of Mathematics \\
  University of Illinois \\
  Urbana, IL, USA \newline
  {\tt \url{https://li-yunsheng.github.io}}}
\email{yl202@illinois.edu}

\date{5 October 2026}

\begin{abstract}  
  We prove that the colored Jones polynomial of a knot does not determine the
  Links--Gould polynomial (associated with the 4-dimensional representation of the
  Lie superalgebra $\fsl(2|1)$) 
  with two different approaches: the first is to find a concrete genus 2 mutation pair
  with different Links--Gould polynomials, and the second follows from the fact that the 
  colored Jones polynomial is invariant under Conway mutation whereas the colored
  Links--Gould polynomial is not. The first approach also applies to a related
  family of polynomial invariants of Kashaev--Mangazeev, which are conjectured to
  be the $\fsl(n|1)$-knot invariants colored by the $2^n$-dimensional
  irreducible representation.
\end{abstract}

\maketitle

{\footnotesize
\tableofcontents
}


\section{Introduction}
\label{sec.intro}

This short paper concerns possible relations between the quantum knot polynomials
associated to representations of simple Lie (super)algebras. Ben-Michael Kohli
asked the first author whether the colored Jones polynomial determines the
Links--Gould (LG) polynomial. Our aim is to explain the answer with two different
arguments, one of which uses genus 2 mutation and the other uses Conway mutation.

\begin{theorem}
\label{thm.1}
The colored Jones polynomial of a knot does not determine the LG-polynomial.  
\end{theorem}

In a similar spirit, we can replace LG polynomial in Theorem~\ref{thm.1} with a
related family, the Kashaev--Mangazeev polynomial $\KM_n$ defined using the
$2^n$-dimensional Nichols algebra as in~\cite{KM} for $n \geq 1$. The $\KM_n$
polynomial is conjectured to equal the quantum knot polynomial associated with
the $2^n$-dimensional representation of the $\fsl(n|1)$ Lie superalgebra
(this conjecture is known to be true for $n=2$, hence $\KM_2$ is the
Links--Gould polynomial).

\begin{theorem}
\label{thm.2}
The colored Jones polynomial of a knot does not determine the $\KM_n$ polynomial
for $n=2,3$. 
\end{theorem}

Let us first explain the zoo of the knot polynomials. Basically, for every
irreducible representation of a Lie (super)algebra, one can associate a knot
polynomial via the Reshetikhin--Turaev functor~\cite{RT:ribbon,Tu:book}. 

The best known quantum knot polynomial is the Jones polynomial~\cite{Jones},
which is associated to the 2-dimensional fundamental representation of the
simple Lie algebra $\fsl_2$. A stronger and more structured version is the colored
Jones polynomial, which encodes the Jones polynomial of a knot and its cables;
alternatively, it is the knot polynomial associated with each finite-dimensional
irreducible representation of $\fsl_2$.
Similarly, the $n$-colored LG polynomial is the knot polynomial associated with
the $4n$-dimensional irreducible representation $V(n-1,a)$ of $\fsl(2|1)$;
see~\cite{LG} and also~\cite{Ga:colored} for a detailed discussion.

The colored Jones polynomial is invariant under Conway mutation. The reason is that
if $V_n$ is the $n$-dimensional irreducible representation of $\fsl_2$, then
$V_n \otimes V_n$ is a direct sum of irreducible representations with multiplicity 1
and this ensures invariance under Conway mutation; see e.g.~\cite{Morton} and
also~\cite{DG}.
For the same reason, the LG polynomial is invariant under Conway mutation
since $V(0,a) \otimes V(0,a)$ decomposes as a direct sum of a 4-dimensional,
a different 4-dimensional and an 8-dimensional representation of $\fsl(2|1)$.

In a sense, genus 2 mutation is a generalization of Conway mutation \cite{DG}.
Cabled mutation is a special case of genus 2 mutation \cite[Lem. 2.12]{DG}, where,
in analogy to Conway mutation, one mutates parallels of a $(2,2)$-tangle in a
string-preserving manner. Let $K(n,m)$ denote the $(n,m)$-parallel of a knot $K$;
it is straightforward to see that if $K$ and $K'$ are cabled mutant to each other,
then $K(n,m)$ and $K'(n,m)$ are also cabled mutant to each other. 

The colored Jones polynomial is invariant under genus 2 mutation \cite[Thm. 3.2]{DG};
this follows from the fact that $V_n \otimes V_n$ is a multiplicity-free sum of
irreducible representations for all $n$. On the other hand, $V(n,a) \otimes V(n,a)$
is not a multiplicity-free sum for all $n \geq 1$ (see~\cite[Lem.3.9]{Ga:colored}),
hence the LG polynomial may fail to be genus 2 mutation invariant.

In fact, the LG polynomial is not even cabled mutation invariant, which we show
in Section~\ref{sub.concrete}.

\section{Proofs}
\label{sec.proofs}

We now give two proofs of Theorem~\ref{thm.1}.

\subsection{By computation on a cabled mutation pair}
\label{sub.concrete}

The following pair of cabled mutant knots from~\cite[Fig.3.9]{DG}
\be
\label{cable-pair}
(14n22185, 14n22589)
\ee
have equal LG polynomials but different 2-colored LG polynomials. Let $\LG_{n,K}$
denote the $n$-colored LG polynomial of a knot $K$ and recall from \cite{Ga:colored}
(where the $V_n$ polynomial is $\LG_n$) that
\be
\label{eq:lg-cable}
\LG_{2,K}(t^2, q^2) = c_{2,0}(t,q) \LG_{1,K(2,1)}(t,q)
+ c_{2, -1}(t,q) \LG_{1,K} (t^2q^{-1}, q) 
+ c_{2,1}(t,q)\LG_{1,K}(t^2q,q),
\ee
where the explicit values of $c_{i,j}$ involved are available at
\cite[eq. (19)]{GL:patterns}. \eqref{eq:lg-cable} thus implies that 
\begin{equation}
  \label{eq:lg-ineqv}
  \LG_{1,14n22185(2,1)}(t,q) \neq \LG_{1,14n22589(2,1)}(t,q),
\end{equation}
while they have the same Jones polynomial, being cabled mutant to each other. We
also confirm \eqref{eq:lg-ineqv} by direct computation with the explicit values
attached in Section~\ref{sec.compute}. This finishes our first proof of
Theorem~\ref{thm.1}. We note here that the pair \eqref{cable-pair} can be replaced
with any other pair in \cite[Fig.3.9]{DG}.

\begin{figure}[htpb!]
  \centering
  \includegraphics[scale=0.35]{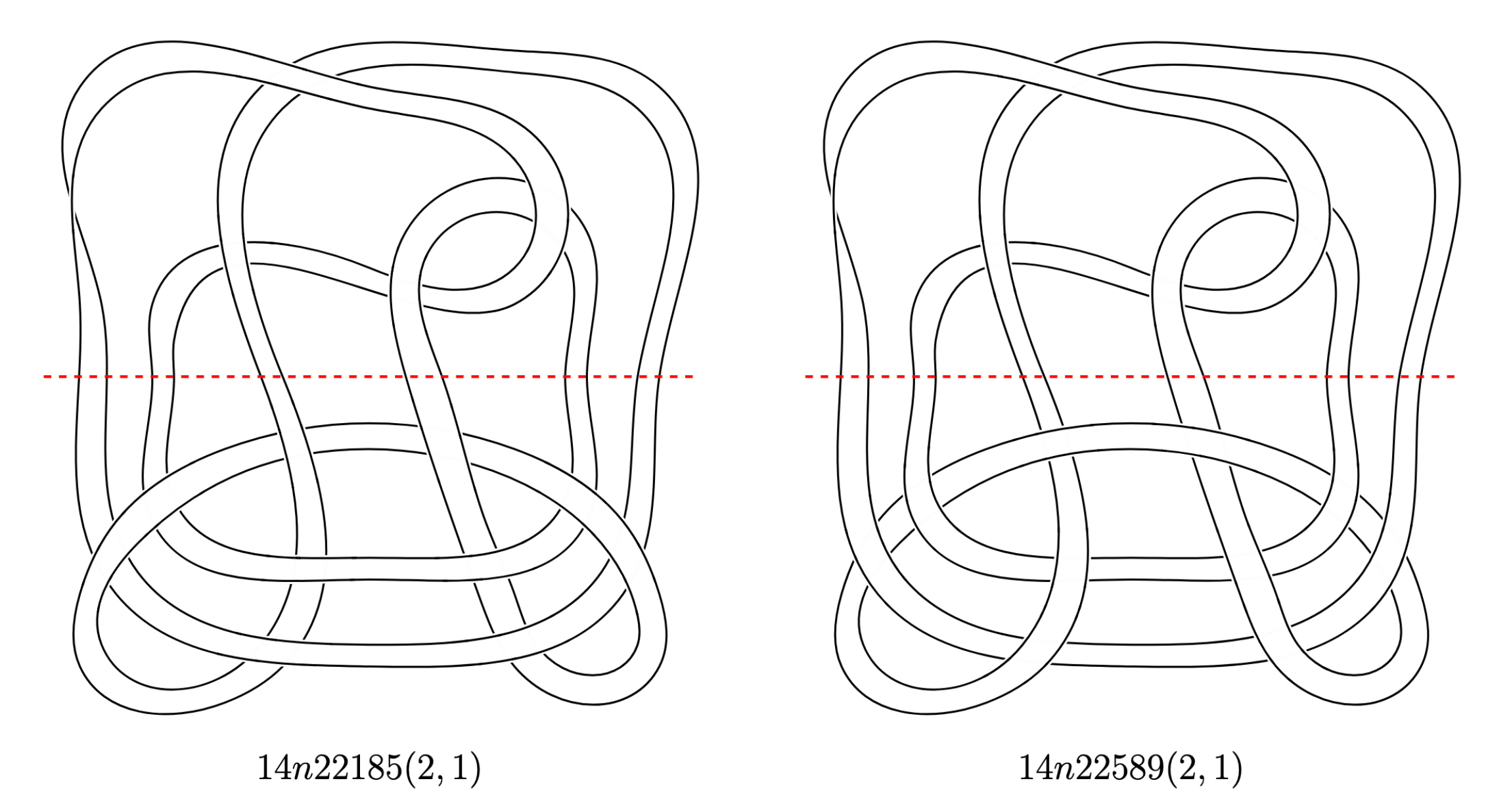}
  \quad 
\caption{The cabled mutation pair $(14n22185(2,1), 14n22589(2,1))$. }
  \label{fig:14n22185_21_14n22589_21}
\end{figure}

We now turn our attention to the Kashaev--Mangazeev polynomial on the same cabled
mutation pair. For a given integer $n\geq1$ and a knot $K$,
$\KM_{n,K}\in \BZ[t^{\pm 1}, q^{\pm 1}]$. It is conjectured that $\KM_n$ is the
quantum knot polynomial associated with the $2^n$-dimensional representation
$V(0,a)$ of $\fsl(n|1)$; this is known for $n=2$,
for which we have $\KM_{2}(t/q, q^2) = \LG_{1}(t,q)$. 
The above mentioned $2^n$-dimensional representation of $\fsl(n|1)$ has
multiplicity-free decomposition of its second tensor power, hence $\KM_n$
is expected to be Conway mutation invariant. On the other hand, 
it is not expected to be genus 2 mutation invariant, 
as we've already observed for $n=2$ from the computation of $\LG_1$.

For $\KM_3$, its R-matrix is $8$ times larger than $\LG_1$, 
which makes its computation too expensive for the two knots in Figure~\ref{fig:14n22185_21_14n22589_21}; 
instead, we compute it for the (2,1)-parallels of the two knots $11n34$ and $11n42$.
$11n34$ and $11n42$ are mutants but not string-preserving mutants, 
hence $11n34(2,1)$ and $11n42(2,1)$ are not cabled mutants; 
nonetheless they are still genus 2 mutants, as the genus 2 mutation surfaces of $11n34$ and $11n42$ 
induce the genus 2 mutation surfaces for their satellites. 
Hence $11n34(2,1)$ and $11n42(2,1)$ have the same colored Jones polynomial. 

The following proposition thus follows
from direct computation explained in Section~\ref{sec.compute}.

\begin{proposition}
\label{prop.2}
The $\KM_n$ polynomials of the pair $(11n34(2,1), 11n42(2,1))$
are not equal for $n=2,3$.
Hence, the colored Jones polynomial does not determine the $\KM_n$ polynomials
for $n=2,3$.
\end{proposition}

\subsection{By Conway mutation and TQFTs}
\label{sub.tqft}

We give a second proof for Theorem~\ref{thm.1} using only Conway mutation
and basic principles of TQFTs.

\begin{lemma}
\label{lem.TQFT}  
If $K$ is a knot and $C$ is a pattern in the solid torus $S^1 \times D^2$, then the
colored Jones polynomial of $K$ determines the colored Jones polynomial of the
satellite $K(C)$.
\end{lemma}

\begin{proof}
The Kauffman bracket skein module of the solid torus is a polynomial algebra in
one generator (the core of the solid torus) over the field $\BQ(q^{1/4})$ and a
$\BQ(q^{1/4})$-vector space, it has a basis that consists of the Jones--Wenzl
idempotents~\cite{Lickorish}. Thus, coloring $C$ with an arbitrary color is
an element of the Kauffman bracket of the solid torus.
\end{proof}

\begin{corollary}
If the colored Jones polynomial of a knot determines the LG-polynomial, then
it determines the colored LG-polynomial.  
\end{corollary}

\begin{proof}
By Lemma~\ref{lem.TQFT}, the colored Jones polynomial of $K$ determines the
colored Jones polynomial of its $(n,1)$-parallel $K(n,1)$ for all $n \geq 1$,
and (under the assumptions of the corollary), it determines the LG polynomial of
$K(n,1)$ for all $n$. The cabling formula of~\cite[Rem.4.1]{Ga:colored} implies that
the LG polynomial of $K(n,1)$ for all $n$ determines the colored Links--Gould
polynomial of $K$.
\end{proof}  

It follows that if the colored Jones polynomial of a knot determines the
LG-polynomial, then the colored LG-polynomial should be invariant under mutation;
however, $\LG_2$ is not, hence the contradiction proves Theorem~\ref{thm.1}.

Finally, one can ask a converse question.

\begin{question}
\label{que.1}
Does the colored Links--Gould polynomial determine the Jones polynomial of a knot?
\end{question}
We think the answer is no, based on loop-expansion considerations,
but we cannot prove this.

\section{Computation}
\label{sec.compute}

In this section we exhibit the usage of code for the computation mentioned in
Section~\ref{sub.concrete} and present some explicit values of polynomials. 
The computation was performed using the method explained in~\cite{GL:patterns}
and~\cite{Li:FPT-RT} and its Python implementation by the second author in
\texttt{SnapPy}~\cite{snappy}, which is currently available with the in-development
version of \texttt{Spherogram}~\cite{spherogram}. 

The code below checks the equalities of values of $\LG_1$ and $\LG_2$ polynomials
of the pair of cabled mutant knots in \eqref{cable-pair}. 


\begin{lstlisting}
sage: import snappy
sage: K1, K2 = snappy.Link('14n22185'), snappy.Link('14n22589')
sage: LG1_K1 = K1.colored_links_gould_polynomial(1) 
sage: LG1_K2 = K2.colored_links_gould_polynomial(1)
sage: LG1_K1 == LG1_K2
True
sage: LG2_K1 = K1.colored_links_gould_polynomial(2)
sage: LG2_K2 = K2.colored_links_gould_polynomial(2)
sage: LG2_K1 == LG2_K2
False
\end{lstlisting}

For obtaining the (0-framed) (2,1)-parallels of the two knots we use the code in \cite{knotsatellite}:

\begin{lstlisting}
sage: from parallel import *
sage: K1_21 = parallel(2, 1, K1)
sage: K2_21 = parallel(2, 1, K2)
\end{lstlisting}

The following checks that they have equal Jones but different LG polynomials:

\begin{lstlisting}
sage: jones_K1_21 = K1_21.colored_jones_polynomial(1)
sage: jones_K2_21 = K2_21.colored_jones_polynomial(1)
sage: jones_K1_21 == jones_K2_21
True
sage: LG_K1_21 = K1_21.colored_links_gould_polynomial(1)
sage: LG_K2_21 = K2_21.colored_links_gould_polynomial(1)
sage: LG_K1_21 == LG_K2_21
False
\end{lstlisting}

The value of the Jones polynomial is as follows:
\begin{tiny}
\begin{equation}
  \begin{aligned}
    J_{14n22185(2,1)}(q) = & q^{-19} - q^{-18} - q^{-16} + q^{-14} - q^{-11} + q^{-10} + 2q^{-9} - 2q^{-8} - 2q^{-7} + 3q^{-6} - q^{-5} 
    \\ & - q^{-4} + 2q^{-3} + 1 - 3q^3 + 2q^4 - 3q^6 + 2q^7 + 2q^8 - q^9 + q^{12} - q^{14} - q^{16} + q^{17}
    \\ = &  J_{14n22589(2,1)}(q)
      \end{aligned}
\end{equation}
\end{tiny}
The explicit values of LG polynomials are rather long. To shorten
our presentation, we use the following abbreviation: if
$f(t,q) \in \BZ[t^{\pm 1},q^{\pm 1}]$ satisfies $f(t,q)=f(t^{-1},q)$ then we write it in the form
\be
f(t,q) = f^{[0]}(q) + \sum_{k=1}^d f^{[k]}(q) (t^k + t^{-k})
\ee
and only give the coefficients $f^{[k]}(q) \in \BZ[q^{\pm 1}]$ for $k=0,1,\dots,d$. 

With the convention above, we have:
    
\begin{tiny}
\begin{equation}
  \begin{aligned}
    \LG_{1, 14n22185(2,1)}^{[0]}(q) = &2 q^{-24} + 2 q^{-22} - 4 q^{-20} - 6 q^{-18} + 6 q^{-14} + 2 q^{-12} + 8 q^{-8} - 10 q^{-6} - 7 + 6 q^{2} \\
    & - 4 q^{4} + 8 q^{6} + 2 q^{8} - 2 q^{10} + 6 q^{12} - 4 q^{14} - 6 q^{16} - 2 q^{18} + 2 q^{20} + 2 q^{22}\\
\LG_{1, 14n22185(2,1)}^{[1]}(q) = &-2 q^{-23} + 4 q^{-19} + q^{-17} - 2 q^{-15} - 4 q^{-13} + 5 q^{-11} - 9 q^{-7} + 15 q^{-5} - 14 q^{-3} + 9 q^{-1} \\
&- 6 q^{3} + 8 q^{5} - 12 q^{7} + 12 q^{9} - 7 q^{11} - 2 q^{13} + 4 q^{15} + q^{17} - q^{19} - q^{21} + 2 q^{23} - q^{27}\\
\LG_{1, 14n22185(2,1)}^{[2]}(q) = &q^{-18} + q^{-16} - 3 q^{-12} - 5 q^{-10} + 3 q^{-8} + 3 q^{-4} + 4 q^{-2} - 3 + 3 q^{2} - q^{4} - 2 q^{6} \\
&- 3 q^{8} - 3 q^{10} + 3 q^{12} + 2 q^{14} + q^{16} + 2 q^{18} - 3 q^{22} - 2 q^{24} + q^{26} + q^{28}\\
\LG_{1, 14n22185(2,1)}^{[3]}(q) = &-q^{-17} + q^{-15} - q^{-13} + 2 q^{-11} + 3 q^{-9} - 7 q^{-7} + 6 q^{-5} - 4 q^{-3} - q^{-1} + 4 q - 5 q^{3} \\
&+ 4 q^{5} - 2 q^{7} + 2 q^{9} + q^{11} - 3 q^{13} + q^{15} - 2 q^{19} + q^{21} + 2 q^{23} - q^{27}\\
\LG_{1, 14n22185(2,1)}^{[4]}(q) = &-q^{-16} + 2 q^{-12} + q^{-8} - 4 q^{-4} + q^{-2} - 1 - q^{2} + 2 q^{4} + 3 q^{8} - q^{12} - 2 q^{16} + q^{20}\\
\LG_{1, 14n22185(2,1)}^{[5]}(q) = &q^{-15} - 2 q^{-13} + 3 q^{-9} - 4 q^{-7} + 2 q^{-5} + q^{-3} - 2 q^{-1} + 3 q - 3 q^{3} + 2 q^{5} \\
&- q^{7} - 2 q^{9} + 4 q^{11} - 3 q^{13} + 2 q^{17} - q^{19}\\
\LG_{1, 14n22185(2,1)}^{[6]}(q) = &q^{-14} - 3 q^{-10} + 2 q^{-6} + q^{-2} - q^{6} - 2 q^{10} + 3 q^{14} - q^{18}\\
\LG_{1, 14n22185(2,1)}^{[7]}(q) = &-q^{-13} + q^{-11} + 2 q^{-9} - 2 q^{-7} - q^{-1} + q - q^{3} + q^{5} + 2 q^{11} - 2 q^{13} - q^{15} + q^{17}
  \end{aligned}
\end{equation}

\begin{equation}
  \begin{aligned}
    \LG_{1, 14n22589(2,1)}^{[0]}(q) = &2 q^{-24} + 2 q^{-22} - 4 q^{-20} - 8 q^{-18} + 10 q^{-14} + 2 q^{-12} + 2 q^{-10} + 8 q^{-8} - 16 q^{-6} \\
    &- 2 q^{-2} - 7 + 8 q^{2} - 4 q^{4} + 14 q^{6} + 2 q^{8} - 4 q^{10} + 6 q^{12} - 8 q^{14} - 6 q^{16} + 2 q^{20} + 2 q^{22}\\
\LG_{1, 14n22589(2,1)}^{[1]}(q) = &-2 q^{-23} + 4 q^{-19} + 3 q^{-17} - 6 q^{-15} - 3 q^{-13} + 10 q^{-11} - 10 q^{-9} - 3 q^{-7} + 17 q^{-5} - 18 q^{-3} \\
& + 16 q^{-1} - 7 q - q^{3} + 5 q^{5} - 16 q^{7} + 20 q^{9} - 15 q^{11} + 9 q^{15} - 2 q^{17} - q^{19} - q^{21} + 2 q^{23} - q^{27}\\
\LG_{1, 14n22589(2,1)}^{[2]}(q) = &q^{-18} + 3 q^{-16} - q^{-14} - 9 q^{-12} - 2 q^{-10} + 7 q^{-8} - 2 q^{-6} + 5 q^{-4} + 3 q^{-2} - 3 \\
&+ 3 q^{2} - 3 q^{4} - q^{6} - 7 q^{8} - q^{10} + 9 q^{12} - q^{14} - q^{16} + 3 q^{18} - 3 q^{22} - 2 q^{24} + q^{26} + q^{28}\\
\LG_{1, 14n22589(2,1)}^{[3]}(q) = &-q^{-17} - 2 q^{-15} + 3 q^{-13} + 6 q^{-11} - 4 q^{-9} - 3 q^{-7} + 4 q^{-5} - 7 q^{-3} + 3 q^{-1} - q \\
&+ 2 q^{5} - q^{7} + 8 q^{9} - 7 q^{11} - 2 q^{13} + 3 q^{15} - 2 q^{17} - q^{19} + q^{21} + 2 q^{23} - q^{27}\\
\LG_{1, 14n22589(2,1)}^{[4]}(q) = &-q^{-4} + q^{-2} - q^{2} + q^{4}\\
  \end{aligned}
\end{equation}
\end{tiny}

Likewise, using the code in \cite{sln1poly},
one can compute the $\KM_n$ polynomials for the ($-1$)-framed (2,1)-parallels 
of $11n34$ and $11n42$ (so that no additional half-twist is added) for $n=2,3$ as follows and
check that they differ as asserted in Proposition~\ref{prop.2}. 

\begin{lstlisting}
sage: import snappy
sage: K1, K2 = snappy.Link('11n34'), snappy.Link('11n42')
sage: from parallel import *
sage: K1_21 = parallel(2, 1, K1, framing=-1)
sage: K2_21 = parallel(2, 1, K2, framing=-1)
sage: from sln1 import *
sage: [sln1(n, K1_21) == sln1(n, K2_21) for n in [2,3]]
[False, False]
\end{lstlisting}

The computation of $\KM_3$, however, is very expensive, finishing in days and taking 650~GB of RAM for $11n34(2,1)$, 
as the knot has 45 crossings and the best possible contraction sequence the implementation found
has contraction-width 14. In comparison, the computation for $\KM_2$ took only about 1.5~GB of RAM and finished in a few minutes for both of the paralleled knots. 

$\KM_3$ is not invariant under $t\mapsto t^{-1}$, but it is after the change of variables $\KM_3(t/q,q)$. We therefore give the values of $\KM_3(t/q,q)$ for the two knots as follows, using the same abbreviation as for the LG polynomials above. 
\begin{tiny}
  \begin{equation}
    \begin{aligned}
      \KM_{3, 11n34(2,1)}^{[0]}(q) = &q^{-19} + 8 q^{-18} - 14 q^{-17} + 42 q^{-16} - 20 q^{-15} + 5 q^{-14} - 11 q^{-13} + 4 q^{-12} - 183 q^{-11} + 226 q^{-10} \\
      &- 287 q^{-9} + 221 q^{-8} + 30 q^{-7} - 86 q^{-6} + 355 q^{-5} - 381 q^{-4} + 528 q^{-3} - 441 q^{-2} + 11 q^{-1} - 100 - 121 q \\
      &+ 167 q^{2} - 100 q^{3} + 219 q^{4} - 93 q^{5} + 81 q^{6} - 50 q^{7} + 18 q^{8} - 35 q^{9} + 10 q^{10} - 10 q^{11} + 8 q^{12} - q^{13}\\
\KM_{3, 11n34(2,1)}^{[1]}(q) = &-2 q^{-19} + 3 q^{-18} - 4 q^{-17} + 9 q^{-16} - 16 q^{-15} + 10 q^{-14} - 24 q^{-13} + 40 q^{-12} - 46 q^{-11} + 77 q^{-10} \\
&- 53 q^{-9} + 70 q^{-8} - 26 q^{-7} - 87 q^{-6} + 146 q^{-5} - 205 q^{-4} + 108 q^{-3} - 140 q^{-2} + 106 q^{-1} + 96 - 67 q \\
&+ 84 q^{2} - 70 q^{3} + 11 q^{4} - 23 q^{5} - 5 q^{6} - 3 q^{7} + 14 q^{8} - 17 q^{9} + 22 q^{10} - 10 q^{11} + q^{12} + q^{13}\\
\KM_{3, 11n34(2,1)}^{[2]}(q) = &2 q^{-19} - 7 q^{-18} + 4 q^{-17} - 22 q^{-16} + 9 q^{-14} + 31 q^{-13} - 20 q^{-12} + 105 q^{-11} - 80 q^{-10} + 73 q^{-9} \\
& - 119 q^{-8} + 8 q^{-7} + 8 q^{-6} - 193 q^{-5} + 189 q^{-4} - 169 q^{-3} + 256 q^{-2} - 56 q^{-1} + 72 + 22 q - 127 q^{2}\\
& + 75 q^{3} - 113 q^{4} + 57 q^{5} - 36 q^{6} + 31 q^{7} - 5 q^{8} + 7 q^{9} - 2 q^{10} + 3 q^{11} - 3 q^{12}\\
\KM_{3, 11n34(2,1)}^{[3]}(q) = &q^{-19} - 6 q^{-18} + 7 q^{-17} + 13 q^{-15} - 16 q^{-14} + 22 q^{-13} - 46 q^{-12} + 74 q^{-11} - 147 q^{-10}\\
& + 101 q^{-9} - 11 q^{-8} - 93 q^{-7} + 161 q^{-6} - 138 q^{-5} + 231 q^{-4} - 123 q^{-3} + 66 q^{-2} - 40 q^{-1} - 152 \\
&+ 90 q - 41 q^{2} + 29 q^{3} - 19 q^{4} + 43 q^{5} + 9 q^{6} - 21 q^{7} + 8 q^{8} + 19 q^{9} - 32 q^{10} + 16 q^{11} - 5 q^{12}\\
\KM_{3, 11n34(2,1)}^{[4]}(q) = &-q^{-19} + 5 q^{-18} + 9 q^{-17} - 21 q^{-16} + 25 q^{-15} - 8 q^{-14} - 39 q^{-13} - 3 q^{-12} + 2 q^{-11} + 31 q^{-10} \\
&- 7 q^{-9} + 47 q^{-8} + 19 q^{-7} + 20 q^{-6} - 26 q^{-5} - 48 q^{-4} - 77 q^{-3} - 55 q^{-2} + 87 q^{-1} - 4 + 55 q \\
&+ 28 q^{2} + 19 q^{3} - 5 q^{4} - 61 q^{5} + 16 q^{6} - 21 q^{7} - 2 q^{8} + 9 q^{9} + 2 q^{10} + 7 q^{11} - 3 q^{12}\\
\KM_{3, 11n34(2,1)}^{[5]}(q) = &-q^{-19} - 3 q^{-18} + 3 q^{-17} - 7 q^{-16} + 16 q^{-15} - 10 q^{-14} + 9 q^{-13} + 46 q^{-12} - 70 q^{-11} + 23 q^{-10} \\
&- 30 q^{-9} - 5 q^{-8} + 38 q^{-7} - 60 q^{-6} + 70 q^{-5} - 34 q^{-4} + 44 q^{-3} + 29 q^{-2} - 50 q^{-1} - 1 - 38 q + 21 q^{2} \\
&- 21 q^{3} + 20 q^{4} + 26 q^{5} - 5 q^{6} + 2 q^{7} - 18 q^{8} + 9 q^{9} - q^{10} - 7 q^{11} + 5 q^{12}\\
\KM_{3, 11n34(2,1)}^{[6]}(q) = &q^{-19} + 5 q^{-18} - 16 q^{-17} + 20 q^{-16} - 16 q^{-15} + 2 q^{-14} - 2 q^{-13} + 4 q^{-12} + 20 q^{-11} - 25 q^{-10}\\
& + 51 q^{-9} - 96 q^{-8} + 40 q^{-7} - 32 q^{-6} + 13 q^{-5} + 54 q^{-4} - 26 q^{-3} + 95 q^{-2} - 67 q^{-1} + 33 - 61 q \\
&- 26 q^{2} - 12 q^{3} + q^{4} + 27 q^{5} - 10 q^{6} + 49 q^{7} - 19 q^{8} + 2 q^{9} - 7 q^{10} - 3 q^{11} + q^{12}\\
\KM_{3, 11n34(2,1)}^{[7]}(q) = &-q^{-19} + 5 q^{-18} - 10 q^{-17} + 12 q^{-16} - 21 q^{-15} + 12 q^{-14} + 11 q^{-13} - 28 q^{-12} + 49 q^{-11}\\
& - 42 q^{-10} + 59 q^{-9} - 61 q^{-8} + 44 q^{-7} - 8 q^{-6} - 63 q^{-5} + 70 q^{-4} - 105 q^{-3} + 82 q^{-2} - 68 q^{-1} \\
&+ 72 + 20 q - 32 q^{2} + 60 q^{3} - 58 q^{4} + 18 q^{5} - 31 q^{6} + 13 q^{7} - 8 q^{8} - 2 q^{9} + 16 q^{10} - 4 q^{11} - q^{12}\\
\KM_{3, 11n34(2,1)}^{[8]}(q) = &q^{-19} - 8 q^{-18} + 10 q^{-17} + 2 q^{-16} - 4 q^{-15} + 10 q^{-14} - 6 q^{-13} + 7 q^{-12} - 49 q^{-11} + 7 q^{-10} + 13 q^{-9} \\
&+ 40 q^{-7} + 19 q^{-6} + 27 q^{-5} - 40 q^{-4} - 20 q^{-3} - 34 q^{-2} - 44 q^{-1} + 12 + 33 q + 28 q^{2} + 16 q^{3} + 9 q^{4}\\
& + 2 q^{5} - 28 q^{6} - 10 q^{7} + 10 q^{8} - 9 q^{9} + 4 q^{10} + 2 q^{12}\\
\KM_{3, 11n34(2,1)}^{[9]}(q) = &2 q^{-19} - 2 q^{-18} + 8 q^{-17} - 10 q^{-16} + 2 q^{-15} - q^{-14} - 28 q^{-13} + 20 q^{-12} - 12 q^{-11} + 68 q^{-10} \\
&- 45 q^{-9} + 45 q^{-8} - 31 q^{-7} - 39 q^{-6} + 6 q^{-5} - 68 q^{-4} + 106 q^{-3} - 58 q^{-2} + 107 q^{-1} - 34 \\
&- 16 q - 11 q^{2} - 44 q^{3} + 37 q^{4} - 41 q^{5} + 46 q^{6} - 8 q^{7} + 10 q^{8} - 8 q^{9} - q^{10} + 2 q^{11} - 2 q^{12}\\
\KM_{3, 11n34(2,1)}^{[10]}(q) = &-2 q^{-19} + 2 q^{-18} + q^{-17} - 11 q^{-16} + 13 q^{-15} - 17 q^{-14} + 43 q^{-13} - 22 q^{-12} + 17 q^{-11} + 5 q^{-10}\\
& - 53 q^{-9} + 46 q^{-8} - 109 q^{-7} + 74 q^{-6} - 38 q^{-5} + 39 q^{-4} + 58 q^{-3} - 50 q^{-2} + 84 q^{-1} - 85 \\
&+ 42 q - 45 q^{2} - 16 q^{3} + 19 q^{4} - 19 q^{5} + 31 q^{6} - 19 q^{7} + 16 q^{8} - q^{9} - q^{10} - q^{11} - q^{12}\\
\KM_{3, 11n34(2,1)}^{[11]}(q) = &-q^{-19} - q^{-18} - q^{-17} + 15 q^{-15} - 8 q^{-14} + 15 q^{-13} - 15 q^{-12} - 3 q^{-11} - 25 q^{-10} - 26 q^{-9} \\
&+ 32 q^{-8} - 9 q^{-7} + 67 q^{-6} + 10 q^{-5} + 5 q^{-4} - 17 q^{-3} - 56 q^{-2} - 7 q^{-1} - 32 + 19 q + 34 q^{2} + 11 q^{3} \\
&+ 16 q^{4} - 8 q^{5} - 7 q^{6} - 5 q^{7} - 7 q^{8} + 4 q^{9} - 4 q^{10} + 3 q^{11} + q^{12}\\
    \end{aligned}
  \end{equation}
\end{tiny}
\begin{tiny}
  \begin{equation}
    \begin{aligned}
\KM_{3, 11n34(2,1)}^{[12]}(q) = &q^{-19} + 3 q^{-18} - 4 q^{-17} + 5 q^{-16} - 11 q^{-15} + 7 q^{-14} - 23 q^{-13} + 7 q^{-12} + 7 q^{-11} + q^{-10} + 42 q^{-9} \\
&- 17 q^{-8} + 35 q^{-7} - 55 q^{-6} + 3 q^{-5} - 25 q^{-4} - 17 q^{-3} + 35 q^{-2} - 19 q^{-1} + 49 - 32 q + 24 q^{2} \\
&- 11 q^{3} - 8 q^{4} + 11 q^{5} - 13 q^{6} + 11 q^{7} - 11 q^{8} + 6 q^{9} - 2 q^{10} + q^{11}\\
\KM_{3, 11n34(2,1)}^{[13]}(q) = &q^{-18} - 2 q^{-17} + 5 q^{-16} - 10 q^{-15} + 4 q^{-14} - 4 q^{-13} + 2 q^{-12} + 18 q^{-11} - 19 q^{-10} + 43 q^{-9} \\
&- 45 q^{-8} + 16 q^{-7} - 25 q^{-6} - 21 q^{-5} + 36 q^{-4} - 40 q^{-3} + 74 q^{-2} - 34 q^{-1} + 31 - 9 q - 23 q^{2} \\
&+ 9 q^{3} - 30 q^{4} + 24 q^{5} - 12 q^{6} + 11 q^{7} + q^{8} + q^{9} - 2 q^{11}\\
\KM_{3, 11n34(2,1)}^{[14]}(q) = &-2 q^{-18} + q^{-16} + 2 q^{-15} + 7 q^{-14} - 2 q^{-13} + 6 q^{-12} - 11 q^{-11} - 10 q^{-10} - 6 q^{-9} - 13 q^{-8} + 24 q^{-7} \\
&+ 22 q^{-5} + 6 q^{-4} - 13 q^{-3} + 3 q^{-2} - 26 q^{-1} + 5 - 5 q + 4 q^{2} + 12 q^{3} - 4 q^{4} + 4 q^{5} - 3 q^{6} - q^{8} - q^{9} + q^{10}\\
\KM_{3, 11n34(2,1)}^{[15]}(q) = &q^{-17} - q^{-16} - q^{-15} + 2 q^{-14} - 5 q^{-13} + 9 q^{-12} - 14 q^{-11} + 17 q^{-10} - 13 q^{-9} \\
&+ 6 q^{-8} + 5 q^{-7} - 14 q^{-6} + 23 q^{-5} - 26 q^{-4} + 24 q^{-3} - 21 q^{-2} \\
&+ 6 q^{-1} - 5 q + 11 q^{2} - 8 q^{3} + 11 q^{4} - 8 q^{5} + 4 q^{6} - 4 q^{7} + q^{8} - q^{9} + q^{10}\\
\KM_{3, 11n34(2,1)}^{[16]}(q) = &q^{-17} - q^{-16} + q^{-15} - 4 q^{-14} + 2 q^{-13} - 3 q^{-12} + 3 q^{-11} + 5 q^{-10} - q^{-9} + 8 q^{-8} - 13 q^{-7} + 5 q^{-6} \\
&- 12 q^{-5} + q^{-4} + 6 q^{-3} - 3 q^{-2} + 14 q^{-1} - 9 + 7 q - 7 q^{2} - 2 q^{3} + 2 q^{4} - 3 q^{5} + 4 q^{6} - 2 q^{7} + 2 q^{8} - q^{9}\\
\KM_{3, 11n34(2,1)}^{[17]}(q) = &-q^{-16} + 2 q^{-15} - 2 q^{-14} + 4 q^{-13} - 4 q^{-12} + 3 q^{-11} - 3 q^{-10} - 3 q^{-9} + 5 q^{-8} - 6 q^{-7} + 10 q^{-6} \\
&- 6 q^{-5} + 5 q^{-4} - 3 q^{-3} - 3 q^{-2} + 3 q^{-1} - 4 + 4 q - 2 q^{2} + 2 q^{3} - q^{4}
    \end{aligned}
  \end{equation}
\end{tiny}

\begin{tiny}
  \begin{equation}
    \begin{aligned}
    \KM_{3, 11n42(2,1)}^{[0]}(q) = &q^{-19} + 4 q^{-18} + 22 q^{-16} - 58 q^{-15} + 141 q^{-14} - 189 q^{-13} + 108 q^{-12} - 177 q^{-11} + 226 q^{-10} - 289 q^{-9} \\
&+ 197 q^{-8} + 208 q^{-7} - 312 q^{-6} + 465 q^{-5} - 531 q^{-4} + 534 q^{-3} - 475 q^{-2} + 75 q^{-1} + 36 - 225 q \\
&+ 249 q^{2} - 64 q^{3} + 127 q^{4} - 105 q^{5} + 103 q^{6} - 86 q^{7} + 62 q^{8} - 71 q^{9} + 34 q^{10} - 18 q^{11} + 10 q^{12} - q^{13}\\
\KM_{3, 11n42(2,1)}^{[1]}(q) = &-q^{-19} + 2 q^{-18} - 6 q^{-17} + 21 q^{-16} - 33 q^{-15} + 28 q^{-14} - 56 q^{-13} + 80 q^{-12} - 45 q^{-11} + 7 q^{-10} \\
&+ 21 q^{-9} + 33 q^{-8} - 31 q^{-7} - 131 q^{-6} + 228 q^{-5} - 185 q^{-4} + 73 q^{-3} - 59 q^{-2} + 25 q^{-1} + 86 \\
&- 101 q + 121 q^{2} - 74 q^{3} - 40 q^{4} + 62 q^{5} - 11 q^{6} - 46 q^{7} + 41 q^{8} - 13 q^{9} + 7 q^{10} - 4 q^{11} + q^{13}\\
\KM_{3, 11n42(2,1)}^{[2]}(q) = &q^{-19} - 3 q^{-18} - 29 q^{-16} + 38 q^{-15} - 52 q^{-14} + 99 q^{-13} - 87 q^{-12} + 147 q^{-11} - 94 q^{-10} + 57 q^{-9} \\
&- 100 q^{-8} - 49 q^{-7} + 95 q^{-6} - 275 q^{-5} + 297 q^{-4} - 229 q^{-3} + 306 q^{-2} - 81 q^{-1} + 55 + 31 q - 168 q^{2} \\
&+ 88 q^{3} - 106 q^{4} + 56 q^{5} - 44 q^{6} + 68 q^{7} - 44 q^{8} + 37 q^{9} - 19 q^{10} + 12 q^{11} - 7 q^{12}\\
\KM_{3, 11n42(2,1)}^{[3]}(q) = &-q^{-19} - 6 q^{-18} + 13 q^{-17} - 20 q^{-16} + 50 q^{-15} - 62 q^{-14} + 78 q^{-13} - 73 q^{-12} + 36 q^{-11} - 69 q^{-10} \\
&+ 15 q^{-9} + 46 q^{-8} - 100 q^{-7} + 222 q^{-6} - 207 q^{-5} + 201 q^{-4} - 89 q^{-3} - 23 q^{-2} + 24 q^{-1} - 152 + 135 q \\
&- 49 q^{2} + 16 q^{3} + 42 q^{4} - 31 q^{5} - 6 q^{6} + 33 q^{7} - 37 q^{8} + 24 q^{9} - 13 q^{10} + 4 q^{11} - q^{12}\\
\KM_{3, 11n42(2,1)}^{[4]}(q) = &q^{-19} + 5 q^{-18} - 4 q^{-17} + 12 q^{-16} - 2 q^{-15} - 26 q^{-14} - q^{-13} + 5 q^{-12} - 57 q^{-11} + 61 q^{-10} + 15 q^{-9}\\
& + 17 q^{-8} - 8 q^{-7} + 55 q^{-6} + 6 q^{-5} - 81 q^{-4} + 2 q^{-3} - 84 q^{-2} + 63 q^{-1} - 70 + 104 q + 4 q^{2} - 21 q^{3} \\
&+ 65 q^{4} - 61 q^{5} + 17 q^{6} - 28 q^{7} + 11 q^{8} - 8 q^{9} + 7 q^{10} - q^{11} + 2 q^{12}\\
\KM_{3, 11n42(2,1)}^{[5]}(q) = &q^{-18} - 7 q^{-17} + 8 q^{-16} - 20 q^{-15} + 27 q^{-14} - 14 q^{-13} + 19 q^{-12} - 7 q^{-11} + 29 q^{-10} - 18 q^{-9} \\
&- 41 q^{-8} + 60 q^{-7} - 94 q^{-6} + 18 q^{-5} - 16 q^{-4} + 53 q^{-3} + 40 q^{-2} - 10 q^{-1} + 27 - 41 q - 41 q^{2} \\
&+ 12 q^{3} + 9 q^{4} - 14 q^{5} + 27 q^{6} - 3 q^{7} - 7 q^{9} + 5 q^{10} - 2 q^{11}\\
\KM_{3, 11n42(2,1)}^{[6]}(q) = &-2 q^{-18} + 2 q^{-17} + 2 q^{-16} - 7 q^{-15} + 12 q^{-14} - 4 q^{-13} + 3 q^{-12} + 7 q^{-11} - 37 q^{-10} + 43 q^{-9} \\
&- 58 q^{-8} + 24 q^{-7} + 6 q^{-6} + 10 q^{-5} + 38 q^{-4} - 50 q^{-3} + 54 q^{-2} - 61 q^{-1} + 37 - 28 q + 4 q^{2} - 2 q^{3} \\
&- 19 q^{4} + 39 q^{5} - 30 q^{6} + 25 q^{7} - 5 q^{8} + 2 q^{9} - 5 q^{10}\\
    \end{aligned}
  \end{equation}
\end{tiny}
\begin{tiny}
  \begin{equation}
    \begin{aligned}
\KM_{3, 11n42(2,1)}^{[7]}(q) = &2 q^{-17} - 3 q^{-16} + 2 q^{-15} - q^{-13} - 13 q^{-12} + 27 q^{-11} - 34 q^{-10} + 27 q^{-9} - 10 q^{-8} + 11 q^{-7} \\
&+ 17 q^{-6} - 32 q^{-5} + 42 q^{-4} - 69 q^{-3} + 36 q^{-2} - 35 q^{-1} + 23 + 7 q + 2 q^{2} + 32 q^{3} - 37 q^{4} \\
&+ 22 q^{5} - 21 q^{6} + 5 q^{7} - 4 q^{8} + 2 q^{9} + 2 q^{10}\\
\KM_{3, 11n42(2,1)}^{[8]}(q) = &-q^{-16} - q^{-15} + 5 q^{-14} - 5 q^{-13} + 4 q^{-12} - 6 q^{-11} + 6 q^{-10} + 3 q^{-9} + 4 q^{-8} - 6 q^{-7} - 12 q^{-6} + 2 q^{-5} \\
&- 6 q^{-4} + 13 q^{-3} - 5 q^{-2} + 16 q^{-1} - 8 + 2 q - 2 q^{2} - 4 q^{3} + 4 q^{4} - 7 q^{5} + 6 q^{6} - 6 q^{7} + 5 q^{8} - q^{9}\\
\KM_{3, 11n42(2,1)}^{[9]}(q) = &q^{-16} - 9 q^{-13} + 15 q^{-12} - 16 q^{-11} + 18 q^{-10} - 14 q^{-9} + 6 q^{-8} + 3 q^{-7} - 15 q^{-6} + 28 q^{-5} - 34 q^{-4} \\
&+ 31 q^{-3} - 23 q^{-2} + 17 q^{-1} - 6 - 6 q + 12 q^{2} - 20 q^{3} + 15 q^{4} - 9 q^{5} + 10 q^{6} - 4 q^{7} + q^{8} - q^{9}\\
\KM_{3, 11n42(2,1)}^{[10]}(q) = &-q^{-16} + q^{-15} - q^{-14} + 4 q^{-13} - q^{-12} - 2 q^{-10} - 5 q^{-9} + 5 q^{-8} - 6 q^{-7} + 8 q^{-6} - 2 q^{-5} \\
&+ 3 q^{-4} + 3 q^{-3} - 7 q^{-2} + 4 q^{-1} - 9 + 4 q + 2 q^{3} + 3 q^{4} - 3 q^{5} + q^{6} - 2 q^{7} + q^{8}\\
\KM_{3, 11n42(2,1)}^{[11]}(q) = &q^{-15} - 2 q^{-14} + 2 q^{-13} - 4 q^{-12} + 4 q^{-11} - 2 q^{-10} + 2 q^{-9} + 2 q^{-8} - 5 q^{-7} + 6 q^{-6} - 8 q^{-5} \\
&+ 6 q^{-4} - 5 q^{-3} + 2 q^{-2} + 2 q^{-1} - 2 + 4 q - 4 q^{2} + 2 q^{3} - 2 q^{4} + q^{5}\\
    \end{aligned}
  \end{equation}
\end{tiny}



\bibliographystyle{hamsalpha}
\bibliography{biblio}
\end{document}